\RequirePackage{ifthen}
\newboolean{IPCO}
\setboolean{IPCO}{true}

\ifthenelse {\boolean{IPCO}}
{
	\documentclass[runningheads]{llncs}
} {
	\documentclass[11pt]{article}
	\usepackage[margin=1.3in]{geometry}
	\usepackage{amsthm}
}

\ifthenelse {\boolean{IPCO}}
{
} {
	\newtheorem{theorem}{Theorem}
	\newtheorem{lemma}{Lemma}
	\newtheorem{corollary}{Corollary}
	\newtheorem{definition}{Definition}
	\newtheorem{proposition}{Proposition}

	\newtheorem{observation}{Observation}
}

\newcommand{\conf}[1]{}
\newcommand{\full}[1]{}
\newcommand{\journal}[1]{#1}

\usepackage[makeroom]{cancel}

\usepackage{graphicx}
\usepackage{amsmath,amssymb}
\usepackage{amsfonts}
\usepackage{amssymb,bbm}
\usepackage{mathtools}
\usepackage{hyperref}
\usepackage[utf8]{inputenc}
\usepackage{color}
\usepackage{tikz} 
\usepackage{mathrsfs}
\usepackage{tkz-euclide}
\usepackage{algorithm}
\usepackage[noend]{algpseudocode}
\usepackage{enumerate}
\usepackage{environ,tabularx}

\newcommand{\conv}{\operatorname{conv}}

\newcommand{\R}{\mathbb R}
\newcommand{\N}{\mathbb N}
\newcommand{\Z}{\mathbb Z}
\newcommand{\bfe}{\mathbf{e}}
\newcommand{\bzero}{\mathbf{0}}

\newcommand{\transp}{\top}
\newcommand{\Halmos}{\qed}
\newcommand{\NP}{\texttt{NP}}
\newcommand{\coNP}{\texttt{co}-\texttt{NP}}
\newcommand{\XP}{\texttt{XP}}
\newcommand{\DP}{\texttt{D$^\text{p}$}}

\newcommand{\bZ}{\mathbb{Z}}
\newcommand{\bR}{\mathbb{R}}

\spnewtheorem{myclaim}{Claim}{\bfseries}{\rmfamily}

\allowdisplaybreaks

\begin{document}
	\mainmatter              

	\title{The Complexity of Recognizing Facets for the Knapsack Polytope}
	\author{Rui Chen$^1$, Haoran Zhu$^2$}
	
	\authorrunning{Chen, Zhu}
	
	\institute{$^1$School of Data Science, The Chinese University of Hong Kong, Shenzhen (\email{rchen@cuhk.edu.cn})\\
		$^2$Microsoft (\email{haoranzhu@microsoft.com})
	}
	
	\maketitle
	
	\begin{abstract}
		The complexity class {\DP}  is the class of all languages that are the intersection of a language in {\NP} and a language in \coNP. It was conjectured that recognizing a facet for the knapsack polytope is \DP-complete. We provide a positive answer to this conjecture. Moreover, despite the \DP-hardness of the recognition problem, we give a polynomial time algorithm for deciding if an inequality with a fixed number of distinct coefficients defines a facet of a knapsack polytope.		
	\end{abstract}

	\section{Introduction}\label{sec:intro}
	\noindent The polyhedral approach has been crucial for the success of solving combinatorial optimization (CO) problems of practical sizes over the last few decades. Many important CO problems can be reformulated as a linear optimization problem over certain discrete sets of vectors. The study of the convex hulls of such discrete sets is a central topic in polyhedral combinatorics as it leads to linear programming reformulations of these CO problems. The convex hulls of such discrete sets associated with some of the well-studied CO problems such as the \emph{travelling salesman problem} (TSP), the \emph{clique problem} and the \emph{knapsack problem} are called the {TSP polytope}, the {clique polytope} and the {knapsack polytope} (KP), respectively. 
	The characterization of the facets (i.e., faces corresponding to irredundant valid linear inequalities) of these polytopes is of particular interest in {polyhedral combinatorics} \cite{schrijver2003combinatorial}.  
	However, a complete list of the irredundant linear inequalities describing the combinatorial polytope is generally hard to obtain. Karp and Papadimitriou \cite{karp1982linear} show that, unless \NP=\coNP, there does not exist a computationally tractable description by linear inequalities of the polyhedron associated with any \NP-complete CO problem.
	
	Although it is hard to obtain all facet-defining inequalities in general, from the mixed-integer programming (MIP) perspective, obtaining strong valid inequalities can be critical for reducing the number of nodes required in the branch-and-cut procedure. There has been a very large body of literature aimed at generating valid inequalities for certain combinatorial polytopes \cite{chopra1989spanning,grotschel1979symmetric,marchand2002cutting}. In particular, there has been significant interest in studying valid inequalities for the (0-1) knapsack polytope \cite{balas1975facets,balas1978facets,weismantel19970,del2023multi}, given that knapsack constraints often show up as substructures in general binary integer programs. Moreover, Crowder et al. \cite{crowder1983solving} and Boyd \cite{boyd1994fenchel} have empirically shown that the feasible region of many binary integer programs can be well-approximated by inequalities valid for individual knapsack polytopes. In addition to applications in general MIP solutions, the characterization of knapsack polytopes have applications in other combinatorial problems, e.g., in developing efficient algorithms for the bin packing problem \cite{goemans2020polynomiality,jansen2020structure}.
	
	Due to physical and computational constraints, we may only generate a relatively small set of valid inequalities when solving CO problems in practice, a natural question to ask next is regarding the strength of these inequalities: Given an inequality and an instance of a CO problem, is this inequality facet-defining for the associated combinatorial polytope? We denote this decision problem by \texttt{CO FACETS} where \texttt{CO} is the specific CO problem. Karp, Papadimitriou and Yannakakis are the first ones taking a theoretical perspective to this problem, and studying its computational complexity. For the decision problem \texttt{TSP FACETS}, some results concerning the complexity of this problem are obtained by Karp and Papadimitriou \cite{karp1982linear}. They show that if \texttt{TSP FACETS} is in \NP, then \NP=\coNP. To provide a more general complexity class for the decision problem of recognizing whether an inequality is a facet of a particular polytope, in another seminal paper by Papadimitriou and Yannakakis \cite{papadimitriou1982complexity}, they introduce a new complexity class, \DP, defined as the class of all languages that are the intersection of a language in {\NP} and a language in \coNP. In other words, a decision problem $A$ is in {\DP} if and only if there exist a decision problem $B$ in {\NP} and a decision problem $C$ in {\coNP} such that an $A$ instance has a ``yes" answer if and only if both a corresponding $B$ instance and a corresponding $C$ instance have ``yes" answers.
	An important observation made by \cite{papadimitriou1982complexity} is that, unless {\NP}$=${\coNP}, {\DP} is above {\NP}$\cup${\coNP}, i.e., {\NP}$\cup${\coNP} is a proper sub-class of {\DP}. The complexity class {\DP} is a natural niche for many important classes of problems. For instance, as the motivation problem in \cite{papadimitriou1982complexity}, \texttt{TSP FACETS} is in \DP. This is because, a facet-defining inequality of a polytope $P$ is essentially a valid inequality that holds at equality at $\dim(P)$ affinely independent points in $P$. So determining whether an inequality is facet-defining for a TSP polytope is equivalent to deciding: (i) if this inequality is valid to the polytope ({\coNP} problem), and (ii) if there exist $\dim(P)$ affinely independent points in $P$ that satisfy the inequality at equality ({\NP} problem). Papadimitriou and Yannakakis \cite{papadimitriou1982complexity} show that some other interesting combinatorial problems, including \emph{critical problems}, \emph{exact problems} and \emph{unique solution problems}, are naturally in {\DP}. Some problems were later shown to be complete for {\DP}. In particular, Cai and Meyer \cite{MR882532} show that the \emph{graph minimal 3-colorability problem} is \DP-complete. Rothe \cite{MR1979853} show that the \emph{exact-4-colorability problem} is \DP-complete. Recently, Bulut and Ralphs \cite{bulut2021complexity} show that the optimal value verification problem for inverse MIP is {\DP}-complete. Regarding \texttt{CO FACETS}, in the original paper by Papadimitriou and Yannakakis \cite{papadimitriou1982complexity}, they show that \texttt{CLIQUE FACETS} is \DP-complete, and conjecture the same hardness for \texttt{TSP FACETS}. This conjecture was later proved by Papadimitriou and Wolfe \cite{papadimitriou1985complexity}. When studying the complexity of lifted inequalities for the \emph{knapsack problem}, along with some other interesting results, Hartvigsen and Zemel \cite{hartvigsen1992complexity} show that recognizing valid inequalities for the knapsack polytope is \coNP-complete, and conjectured that \texttt{KNAPSACK FACETS} is \DP-complete. The first main contribution of this paper is that we give a positive answer to this conjecture.
	
	Despite the \DP-completeness of the facet-recognition problem associated with the KP, one can still recognize specific facets of the KP in polynomial time. It has been shown that for an inequality with only binary coefficients on the left-hand side, whether this inequality is facet-defining for a knapsack polytope can be determined in polynomial time \cite{balas1975facets,hammer1975facet,wolsey1975faces}. In this paper, we further extend this result to a more general scenario: as long as the inequality has a fixed number of distinct positive coefficients, the corresponding \texttt{KNAPSACK FACETS} can be solved in polynomial time. In fact, we will show that, \texttt{KNAPSACK FACETS} can be solved in time $n^{K+O(1)}$, where $K$ is the number of distinct positive coefficients of the inequality and $n$ is the dimension. 
	
	The remainder of the paper is organized as follows. In Section~\ref{sec: kp_support}, along with a few auxiliary \DP-complete results, we establish that the recognition problem of a supporting hyperplane for the knapsack polytope is \DP-complete.
	In Section~\ref{sec: kp_facet}, we prove the main result of this paper, which is that recognizing facets for knapsack polytope is also \DP-complete. In Section~\ref{sec: xp}, we give a polynomial time algorithm for \texttt{KNAPSACK FACETS} on inequalities with a fixed number of distinct coefficients.
	\full{Lastly, in Section~\ref{sec: kp_membership}, we show that the problem of recognizing if a given point is in a given knapsack polytope is \NP-complete.}

	\mbox{}\\
	\noindent {\bf Notations. }
	For an integer $n$ we set $[n] := \{1, 2, \ldots, n\}$. We let $\N$ denote the set of positive integers, i.e., $\N=\{1,2,\ldots\}$. For a vector $w \in \R^n$ and $S \subseteq [n]$, we set $w(S): = \sum_{i \in S} w_i$ and $w_S$ being the subvector of $w$ with components $(w_i)_{i\in S}$. For a sequence $f\in \R^\N$ and $S\subseteq \N$ with $|S|<\infty$, we set $f(S): = \sum_{i \in S} f_i$. For $i\in[n]$, we let $\bfe_i$ denote the $i$-th unit vector in $\bR^n$.
	
	\section{Critical Subset Sum and Knapsack Supporting Hyperplane Problems}
	\label{sec: kp_support}
	
	Papadimitriou and Yannakakis \cite{papadimitriou1982complexity} show that the \emph{TSP supporting hyperplane problem}, which is the problem of deciding if a given inequality with integer coefficients provides a supporting hyperplane to the given TSP polytope, is \DP-complete. In this section, we extend the same completeness result to the following {\emph{knapsack supporting hyperplane problem}: Given an inequality $\alpha^\transp x \leq \beta$ with $\alpha \in \Z^n$ and a knapsack set $\left\{x \in \{0,1\}^n: a^\transp x \leq b \right\}$, is it true that this inequality is valid for the associated KP and the corresponding hyperplane has a nonempty intersection with the KP? Throughout the paper, we call the set $\left\{x \in \{0,1\}^n: a^\transp x \leq b \right\}$ with $(a,b)\in \N^{n+1}$ a knapsack set.
		
		Before proceeding to the proof of the main result in this section,  we first introduce a problem in {\DP}.\\
		\emph{Exact vertex cover} (\texttt{EVC}): Given graph $G=(V,E)$ and a positive integer $k$, is it true that the minimum vertex cover of $G$ has size exactly $k$, i.e., there exists $V'$ of size $k$ but no $V'$ of size $k-1$ such that $V'\cap e\neq \emptyset$ for all $e\in E$? We use $(V,E,k)$ to denote one particular instance of \texttt{EVC}.
		
		It has been shown that a class of exact problems, including \texttt{EVC}, are \DP-complete. 
		
		\begin{theorem}[\cite{papadimitriou1982complexity}]
			\label{theo: evc_dpc}
			\texttt{EVC} is \DP-complete.
		\end{theorem}
		
		In this section, we will first define an auxiliary problem, which we call the \emph{critical subset sum problem} (\texttt{CSS}), and show that \texttt{EVC} is reducible to \texttt{CSS} (Theorem~\ref{theo: css_dpc}), and then show that \texttt{CSS} is reducible to the \emph{knapsack supporting hyperplane problem} (Theorem~\ref{theo: supporting_kp_dpc}), thus establishing the \DP-completeness of the \emph{knapsack supporting hyperplane problem}. Here we remark that all reductions we mention in this paper refer to the polynomial time many-one reduction, or Karp reduction \cite{Karp1972}.
		
		Now we define \texttt{CSS}, which is a slight variant of the \emph{subset sum problem}.\\
		\emph{Critical subset sum} (\texttt{CSS}): Given $w \in \Z^n_+$ and a target sum $t$, is it true that there exists a subset $S \subseteq [n]$ such that $w(S) = t-1$, but no subset $T \subseteq [n]$ such that $w(T) = t$? We use $(w, t)$ to denote one particular instance of \texttt{CSS}.
		
		Using the standard reduction from \emph{vertex cover} to \emph{subset sum}, we can show the following result, which will play a crucial rule in the next section.
		\begin{theorem}
			\label{theo: css_dpc}
			\texttt{CSS} is \DP-complete.
		\end{theorem}
		\textbf{Proof.}
		Note that deciding whether there exists $S \subseteq [n]$ such that $w(S) = t-1$ is in \NP, and deciding whether there does not exist $T \subseteq [n]$ such that $w(T) = t$ is in \coNP. Therefore, \texttt{CSS} is in \DP.
		
		By Theorem \ref{theo: evc_dpc}, to show \DP-completeness of \texttt{CSS}, it suffices to show that \texttt{EVC} is reducible to \texttt{CSS}. Given instance $(V,E,k)$ of the exact vertex cover problem, define the following \texttt{CSS} instance. Assume $V=\{v_1,\ldots,v_n\}$ and $E=\{e_1,\ldots,e_m\}$.  For $i=1,\ldots,n$, define $w_i:=1+\sum_{j=1}^m(n+1)^j\mathbbm{1}(v_i\in e_j)$. For $j=1,\ldots,m$, define $w_{n+j}:=(n+1)^j$. Define $t:=n-k+1+\sum_{j=1}^m(n+1)^j$. Then the \texttt{CSS} instance $(w,t)$ has polynomial encoding size with respect to the input size of the \texttt{EVC} instance $(V,E,k)$.
		
		Let $\tilde{I}:=\{i_1,\ldots,i_{p}\}\subseteq [n]$ and $\tilde{V}:=\{v_{i_1},\ldots,v_{i_{p}}\}\subseteq V$. Note that $\tilde{V}$ is a vertex cover of $G$ if and only if the ($q+1$)-th digit of $w(\tilde{I})$ (in base $n+1$) is at least $1$ for $q=1,\ldots,n$. Define $\bar{I}:=[n]\setminus\tilde{I}$ and $\bar{V}=V\setminus\tilde{V}$. Then $\tilde{V}$ is a vertex cover of $G$ if and only if the ($q+1$)-th digit of $w(\bar{I})$ is at most $1$ for $q=1,\ldots,n$. Also note that the first digit of $w(\bar{I})$ is $n-p$. Define $\bar{J}:=\{n+j:j\in[m],~e_j\cap \bar{V}=\emptyset\}$. Then $\tilde{V}$ being a vertex cover of $G$ implies $w(\bar{I}\cup\bar{J})=n-p+\sum_{j=1}^m(n+1)^j$. On the other hand, $w(S)=n-p+\sum_{j=1}^m(n+1)^j$ implies $|[n]\setminus S|=p$ and $\{v_i:i\in [n]\setminus S\}$ being a vertex cover of $G$. Then by definitions of the \texttt{EVC} instance $(V,E,k)$ and the \texttt{CSS} instance $(w,t)$, we have that the \texttt{EVC} instance has a ``yes" answer if and only if the \texttt{CSS} instance has a ``yes" answer.
		\hfill \Halmos
		\journal{
			The above theorem has an immediate corollary on the \emph{exact knapsack problem} (\texttt{EK}) which asks: Given a $n$-dimensional vector $c$, a knapsack constraint $a^\transp x \leq b$ and an integer $L$, is it true that\begin{displaymath}
				\max\{c^\transp x: a^\transp x \leq b, x \in \{0,1\}^n\} = L? 
			\end{displaymath}
			\begin{corollary}
				\texttt{EK} is \DP-complete.
			\end{corollary}
			\textbf{Proof.}
			Note that deciding whether the following inequality is true (i.e., the decision problem form of the knapsack problem) is in \NP:\begin{displaymath}
				\max\{c^\transp x: a^\transp x \leq b, x \in \{0,1\}^n\} \geq L? 
			\end{displaymath}
			Similarly, deciding whether the following inequality is true (i.e., the opposite of the decision problem form of the knapsack problem) is in \coNP:\begin{displaymath}
				\max\{c^\transp x: a^\transp x \leq b, x \in \{0,1\}^n\} \leq L? 
			\end{displaymath}
			Therefore, \texttt{EK} is in \DP. We next show a reduction from \texttt{CSS} to \texttt{EK}.
			Let $(w, t)$ be a \texttt{CSS} instance. Then this instance has ``yes" answer if and only if $\max\{w^\transp x: w^\transp x \leq t, x \in \{0,1\}^n\} = t-1$, which is a ``yes" answer to a particular \texttt{EK} instance.
			\hfill \Halmos
		}		
		We can use Theorem \ref{theo: css_dpc} to derive the \DP-completeness of the \emph{knapsack supporting hyperplane problem}. 
		
		\begin{theorem}
			\label{theo: supporting_kp_dpc}
			The knapsack supporting hyperplane problem is \DP-complete.
		\end{theorem}
		\textbf{Proof.}
		By Theorem~\ref{theo: css_dpc}, it suffices to establish that \texttt{CSS} is reducible to the \emph{knapsack supporting hyperplane problem}. Given a \texttt{CSS} instance $(w, t)$, consider the following instance of the \emph{knapsack supporting hyperplane problem}: Given an inequality $\sum_{i=1}^n w_i x_i \leq t-1$, is it true that this inequality is valid for the KP associated with $\{x \in \{0,1\}^n : \sum_{i=1}^n w_i x_i \leq t\}$ and the corresponding hyperplane has a nonempty intersection with the KP? 
		It is easy to see that this \emph{knapsack supporting hyperplane} instance has a ``yes" answer if and only if $\sum_{i=1}^n w_i x_i = t$ has no solution over $x \in \{0,1\}^n$, but there exists $x^* \in \{0,1\}^n$ such that $\sum_{i=1}^n w_i x^*_i = t-1$. This is equivalent to saying that the \texttt{CSS} instance $(w, t)$ has a ``yes" answer.\hfill \Halmos
		
		\section{\DP-Completeness of \texttt{KNAPSACK FACETS}}
		\label{sec: kp_facet}
		
		In this section, we are going to resolve the conjecture raised by Hartvigsen and Zemel \cite{hartvigsen1992complexity}: \texttt{KNAPSACK FACETS} is \DP-complete.
		
		Before proving the main result of this section, we first present some results regarding the following ``shifted'' Fibonacci sequence $(f_i)_{i=1}^\infty$ defined as:
		\begin{align}
			&f_1 = f_2 = f_3 = 1, 
			\label{eq:gu_fibonacci1}\\
			&f_i = f_{i-2} + f_{i-1},\quad i \geq 4.
			\label{eq:gu_fibonacci2}
		\end{align}
		The idea of incorporating the sequence $f$ into the reduction that we will use later to prove the main result, is motivated by the constructive example in \cite{gu1995lifted,chen2021complexity}, where $f$ is used to construct a hard instance for sequentially lifting a cover inequality.
		
		For this particular sequence $f$, we have the following observations, which can be easily verified by induction.
		\begin{lemma}[\cite{chen2021complexity}]
			\label{obs: 1}
			For $j \geq 3, f_j = \sum_{i=1}^{j-2} f_i$. 
		\end{lemma}
		\begin{lemma}[\cite{chen2021complexity}]
			\label{obs: 2}
			For $j\geq 3, \frac{\sqrt{2}-1}{4}\sqrt{2}^j \leq f_j \leq 2^j$. 
		\end{lemma}
		The sequence $f$ also has the following nice property.
		\begin{lemma}[\cite{chen2021complexity}]
			\label{lem: f_property}
			Let $f$ be defined as in \eqref{eq:gu_fibonacci1}-\eqref{eq:gu_fibonacci2} and $r \geq 1$ be a given integer. For any $\tau \in \Z_+$ satisfying $0 \leq \tau \leq \sum_{i=1}^{2r+1}f_i$, there exists a subset $S \subseteq [2r+1]$ such that $f(S) = \tau$. 
		\end{lemma}
		Using the same argument for \texttt{TSP FACETS} in Section \ref{sec:intro}, we have that \texttt{KNAPSACK FACETS} is in \DP. We are now ready to prove the main result of this section.
		\begin{theorem}
			\label{theo: kp_facet_dpc}
			\texttt{KNAPSACK FACETS} is \DP-complete.
		\end{theorem} 
		\textbf{Proof.}
		It suffices to show that \texttt{CSS} is reducible to \texttt{KNAPSACK FACETS}, as \texttt{CSS} is \DP-complete according to Theorem~\ref{theo: css_dpc}. Consider any \texttt{CSS} instance $(w, t)$, i.e., ``is it true that there exists $S \subseteq [n]$ such that $w(S) = t-1$, but there does not exist $T \subseteq [n]$ such that $w(T) = t$?" Without loss of generality, here we assume that $w_i \leq t-1$ for all $i \in [n]$ and $t\geq 2$. 
		
		We next construct a \texttt{KNAPSACK FACETS} instance. Let $L = w([n]),~r = \lceil \log_2(30L+20) - 1\rceil$, and
		\begin{align}
			& a_i = 
			\begin{cases}
				t f_i,        & i =1, \ldots, 2r+1, \\
				t(2L+1)+1, & i = 2r+2,\\
				(t+1)w_{i-2r-2},~
				&  i = 2r+3, \ldots, 2r+n+2, \\
				t f_{2r+1} + t^2 + t(2L+2) + 1,~
				&  i =2r+n+3, \\
				t + 1, &  i = 2r+n+4.
			\end{cases}
			\label{eq:a}
			\\
			& b = t \sum_{i=1}^{2r+1}f_i  + t^2 + t (2L+2) + 1,   \label{eq:b}\\
			& \alpha_i = 
			\begin{cases}
				f_i,        &  i =1, \ldots, 2r+1, \\
				2L+2, &  i = 2r+2,\\
				w_{i-2r-2}, ~&  i = 2r+3, \ldots, 2r+n+2, \\
				f_{2r+1} + t + 2L+1,  &  i =2r+n+3, \\
				0, &  i = 2r+n+4.
			\end{cases}
			\label{eq:alpha}
			\\ 
			& \beta = \sum_{i=1}^{2r+1}f_i  + t + 2L + 1.   \label{eq:beta}
		\end{align}
		Here $N:=2r+n+4$ is the dimension of the vectors $a$ and $\alpha$.
		Consider the following instance of \texttt{KNAPSACK FACETS}: Given an inequality $\alpha^\transp x \leq \beta$ and a KP $\conv(\{x \in \{0,1\}^N: a^\transp x \leq b\})$, is this inequality facet-defining to the KP? 
		It is easy to verify that the input size of this \texttt{KNAPSACK FACETS} instance is polynomial in that of the \texttt{CSS} instance $(w,t)$. 
		To complete the proof of this theorem, we are going to show: there is a ``yes" answer to the \texttt{CSS} instance $(w, t)$ if and only if $\alpha^\transp x \leq \beta$ is a facet-defining inequality to the KP $\conv(\{x \in \{0,1\}^N: a^\transp x \leq b\})$. 
		
		Given the \texttt{CSS} instance and the \texttt{KNAPSACK FACETS} instance, we have the following claim.
		\begin{claim}
			\label{claim: 1}
			$\sum_{i=1}^{2r}f_i > 3L+2$.
		\end{claim}
		\textbf{Proof of claim.}
		The claim follows from
		$\sum_{i=1}^{2r}f_i = f_{2r+2} \geq \frac{\sqrt{2}-1}{4}\sqrt{2}^{2r+2} >  2^{r+1}/10 \geq 2^{\log_2(30L+20)}/10 = 3L+2$,
		where the first equality is from Lemma~\ref{obs: 1}, the second inequality is from Lemma~\ref{obs: 2} and the last inequality is from the definition of $r$.
		\hfill$\diamond$
		
		In order to determine whether the inequality $\alpha^\transp x \leq \beta$ is facet-defining to the KP $\conv(\{x \in \{0,1\}^N: a^\transp x \leq b\})$, where $N = 2r+n+4$, first, we prove that for the restricted variable space where $x_{N-1} = x_{N} = 0$, one similar inequality $\alpha^\transp x \leq \sum_{i=1}^{2r} f_i$ is indeed facet-defining to $\conv(\{x \in \{0,1\}^N: a^\transp x \leq \sum_{i=1}^{2r} t f_i\})$. Notice that only the right-hand-side values of the inequality and knapsack constraint are different from the target. 
		This is shown through a series of the next 3 claims. 
		\begin{claim}
			\label{claim: 2}
			Inequality $\sum_{i=1}^{2r+1} \alpha_i x_i \leq  \sum_{i=1}^{2r} f_i$ is a facet-defining inequality for the KP $\conv(\{x \in \{0,1\}^{2r+1}: \sum_{i=1}^{2r+1} a_i x_i \leq \sum_{i=1}^{2r} t f_i\})$. 
		\end{claim}
		\textbf{Proof of claim.}
		Note that for $\gamma=1,\ldots,r$,
		by our definition in \eqref{eq:a}-\eqref{eq:beta}, both inequalities $\sum_{i=1}^{2\gamma+1} \alpha_i x_i \leq  \sum_{i=1}^{2\gamma} f_i$ and $\sum_{i=1}^{2\gamma+1} a_i x_i \leq \sum_{i=1}^{2\gamma} t f_i$ are equivalent to $\sum_{i=1}^{2\gamma+1} f_i x_i \leq \sum_{i=1}^{2\gamma} f_i$. 
		We prove a stronger version of the claim: For $\gamma=1,\ldots,r$, inequality $\sum_{i=1}^{2\gamma+1} \alpha_i x_i \leq  \sum_{i=1}^{2\gamma} f_i$ is a facet-defining inequality for the KP $\conv(\{x \in \{0,1\}^{2\gamma+1}: \sum_{i=1}^{2\gamma+1} a_i x_i \leq \sum_{i=1}^{2\gamma} t f_i\})$. 
		We proceed by induction on $\gamma$. When $\gamma=1$, the claim is: $x_1 + x_2 + x_3 \leq 2$ is facet-defining for $\conv(\{x \in \{0,1\}^3: x_1 +  x_2 +  x_3 \leq 2\})$, which is obviously true. Assume that this claim is true when $\gamma = R-1$ for some integer $R \in[2,r-1]$: $\sum_{i=1}^{2R-1} f_i x_i \leq \sum_{i=1}^{2R-2} f_i$ is a facet-defining inequality for $\conv(\{x \in \{0,1\}^{2R-1}: \sum_{i=1}^{2R-1} f_i x_i \leq \sum_{i=1}^{2R-2} f_i\})$. So there exists affinely independent points $v_1, \ldots, v_{2R-1}\in\{0,1\}^{2R-1}$, satisfying $\sum_{i=1}^{2R-1} f_i x_i \leq \sum_{i=1}^{2R-2} f_i$ at equality. For $v\in\{0,1\}^{2R-1}$, let $(v, 0, 1)$ denote the binary point in $\{0,1\}^{2R+1}$ obtained by appending to $v$ two new components with values $0$ and $1$. It is then easy to verify that, for any $j \in [2R-1], x=(v_j, 0,1)$ satisfies $\sum_{i=1}^{2R+1} f_i x_i \leq \sum_{i=1}^{2R} f_i$ at equality as $f_{2R+1} = f_{2R} + f_{2R-1}$. Define $p := (1, \ldots, 1, 0) \in \{0,1\}^{2R+1}$, then $\sum_{i=1}^{2R+1} f_i p_i = \sum_{i=1}^{2R} f_i$. Define $q := (0, \ldots, 0, 1,1) \in \{0,1\}^{2R+1}$, then $\sum_{i=1}^{2R+1} f_i q_i = f_{2R} + f_{2R+1} = \sum_{i=1}^{2R} f_i$. Here the last equality is from Lemma~\ref{obs: 1}. Therefore, we have obtained the following $2R+1$ binary points in $\{0,1\}^{2R+1}: (v_1, 0, 1), \ldots, (v_{2R-1}, 0, 1),p, q$, where $v_1, \ldots, v_{2R-1}$ are affinely independent in $\{0,1\}^{2R-1}$. It is easy to see that these $2R+1$ binary points are affinely independent in $\{0,1\}^{2R+1}$, and satisfy $\sum_{i=1}^{2R+1} f_i x_i \leq \sum_{i=1}^{2R} f_i$ at equality. 
		\hfill$\diamond$
		
		\begin{claim}
			\label{claim: 3}
			Inequality $\sum_{i=1}^{2r+2} \alpha_i x_i \leq  \sum_{i=1}^{2r} f_i$ is a facet-defining inequality for the KP $\conv(\{x \in \{0,1\}^{2r+2}: \sum_{i=1}^{2r+2} a_i x_i \leq \sum_{i=1}^{2r} t f_i\})$. 
		\end{claim}
		\textbf{Proof of claim.}
		First, let's verify that $\sum_{i=1}^{2r+2} \alpha_i x_i \leq  \sum_{i=1}^{2r} f_i$ is valid for such KP. When $x_{2r+2} = 0$, it is trivially valid. When $x_{2r+2} = 1$, the knapsack constraint implies that $\sum_{i=1}^{2r+1} t f_i x_i \leq \sum_{i=1}^{2r} t f_i - t(2L+1) - 1$. In this case $\sum_{i=1}^{2r+1} f_i x_i \leq \sum_{i=1}^{2r} f_i - 2L - 2$, which means that $\sum_{i=1}^{2r+1} f_i x_i + (2L+2)x_{2r+2} \leq  \sum_{i=1}^{2r} f_i$ is a valid inequality. Second, from the last Claim~\ref{claim: 2}, it suffices to show that there exists a binary point $x^* \in \{0,1\}^{2r+2}$ with $x^*_{2r+2} = 1$, such that $\sum_{i=1}^{2r+1} a_i x^*_i + a_{2r+2}  \leq \sum_{i=1}^{2r} t f_i$ while $\sum_{i=1}^{2r+1} \alpha_i x^*_i + \alpha_{2r+2} =  \sum_{i=1}^{2r} f_i$. By definitions of $a$ in \eqref{eq:a} and $\alpha$ in \eqref{eq:alpha}, it suffices to find a binary point $x^*$, such that $\sum_{i=1}^{2r+1} f_i x^*_i = \sum_{i=1}^{2r}f_i - 2L - 2$. 
		From the above Claim~\ref{claim: 1}, $\sum_{i=1}^{2r}f_i - 2L - 2 \geq L$. By Lemma~\ref{lem: f_property}, we know that such binary point $x^*$ must exist. 
		\hfill$\diamond$
		
		\begin{claim}
			\label{claim: 4}
			Inequality $\sum_{i=1}^{2r+n+2} \alpha_i x_i \leq  \sum_{i=1}^{2r} f_i$ is a facet-defining inequality for the KP $\conv(\{x \in \{0,1\}^{2r+n+2}: \sum_{i=1}^{2r+n+2} a_i x_i \leq \sum_{i=1}^{2r} t f_i\})$.
		\end{claim}
		\textbf{Proof of claim.}
		By definitions of $a$ in \eqref{eq:a} and $\alpha$ in \eqref{eq:alpha}, we need to show that inequality
		\begin{multline}
			\sum_{i=1}^{2r+1} f_i x_i + \left( 2L+2 \right) x_{2r+2}
			+\sum_{i=2r+3}^{2r+n+2} w_{i-2r-2} x_i \leq  \sum_{i=1}^{2r} f_i
			\label{eq:claim_s_ineq}
		\end{multline}
		is facet-defining for the KP defined by the following knapsack constraint:
		\begin{multline}
			\sum_{i=1}^{2r+1} t f_i x_i + \big( t(2L+1) + 1 \big) x_{2r+2}
			+ \sum_{i=2r+3}^{2r+n+2} (t+1) w_{i-2r-2} x_i \leq \sum_{i=1}^{2r} t f_i.
			\label{eq:claim_s_kp}
		\end{multline}
		First of all, we verify that inequality \eqref{eq:claim_s_ineq} is indeed valid for the KP defined by constraint \eqref{eq:claim_s_kp}. 
		For any binary point $x \in \{0,1\}^{2r+n+2}$, knapsack constraint \eqref{eq:claim_s_kp} implies that 
		\begin{multline*}
			\sum_{i=1}^{2r+1} f_i x_i \leq \sum_{i=1}^{2r}  f_i
			- (2L+1) x_{2r+2}- \sum_{i=2r+3}^{2r+n+2} w_{i-2r-2} x_i
			- \left\lceil \frac{x_{2r+2} + \sum_{i=2r+3}^{2r+n+2} w_{i-2r-2} x_i}{t} \right\rceil.
		\end{multline*}
		Therefore, we have
		\begin{multline*}
			\sum_{i=1}^{2r+1} f_i x_i + \left( 2L+2 \right) x_{2r+2} +  \sum_{i=2r+3}^{2r+n+2} w_{i-2r-2} x_i \\
			\leq\sum_{i=1}^{2r}  f_i + x_{2r+2} - \left\lceil \frac{x_{2r+2} + \sum_{i=2r+3}^{2r+n+2} w_{i-2r-2} x_i}{t} \right\rceil
			\leq \sum_{i=1}^{2r}  f_i,
		\end{multline*}
		i.e., inequality \eqref{eq:claim_s_ineq} is valid.
		To complete the proof, it suffices to show that there exist $2r+n+2$ affinely independent binary points satisfying the knapsack constraint \eqref{eq:claim_s_kp}, on which \eqref{eq:claim_s_ineq} holds at equality. 
		From the last Claim~\ref{claim: 3}, we can find $2r+2$ affinely independent binary points $v_1, \ldots, v_{2r+2}$ in $\{0,1\}^{2r+2}$ satisfying 
		$\sum_{i=1}^{2r+1} f_i x_i + \big( 2L+2 \big) x_{2r+2}= \sum_{i=1}^{2r} f_i$ and $\sum_{i=1}^{2r+1} t f_i x_i + \left( t(2L+1) + 1 \right) x_{2r+2}  \leq \sum_{i=1}^{2r} t f_i.$ It implies that $(v_1, 0, \ldots, 0), \ldots ,(v_{2r+2}, 0, \ldots, 0)$ are affinely independent in $\{0,1\}^{2r+n+2}$, satisfying the knapsack constraint \eqref{eq:claim_s_kp}, and satisfying \eqref{eq:claim_s_ineq} at equality. Now, for each $i \in [n]$, consider $\sum_{i=1}^{2r} f_i - 2L - 2 - w_i$. By Claim~\ref{claim: 1}, we know that $\sum_{i=1}^{2r} f_i - 2L - 2 - w_i \geq 0$. So from Lemma~\ref{lem: f_property}, we can find $x^*$ with $x^*_{2r+2} = x^*_{2r+2+i} = 1$, and $x^*_{2r+2+j} = 0$ for all $j \in [n] \setminus \{i\}$, and $\sum_{i=1}^{2r+1} f_i x^*_i = \sum_{i=1}^{2r} f_i - 2L - 2 - w_i$. Also note that $w_i \leq t-1$. It is then easy to verify that $x^*$ satisfies the knapsack constraint \eqref{eq:claim_s_kp}, and satisfies \eqref{eq:claim_s_ineq} at equality. Therefore, we have found in total $2r+n+2$ binary points that satisfy knapsack constraint \eqref{eq:claim_s_kp}, and satisfy \eqref{eq:claim_s_ineq} at equality. Moreover, these $2r+n+2$ points are obviously affinely independent. 
		\hfill$\diamond$

		Now, we are ready to prove the validity of the reduction: there is a ``yes" answer to the \texttt{CSS} instance $(w, t)$ if and only if $\alpha^\transp x \leq \beta$ is a facet-defining inequality for the KP $\conv(\{x \in \{0,1\}^N: a^\transp x \leq b\})$. 
		
		We first verify that $w(S) \neq t$ for all $S \subseteq [n]$ if and only if inequality $\alpha^\transp x \leq \beta$ is valid for the KP defined by $a^\transp x \leq b$.
		In other words, we would like to show that $w(S) \neq t$ for all $S \subseteq [n]$ if and only if for all $\bar x \in \{0,1\}^N$ with $a^\transp \bar x \leq b$, we have $\alpha^\transp \bar x \leq \beta$. Consider an arbitrary $\bar{x}\in \{0,1\}^N$ with $a^\transp \bar x \leq b$.
		Depending on the values of $\bar x_{N-1}$ and $\bar x_{N}$, we consider the following four cases.
		\begin{enumerate}
			\item[(a)] $\bar x_{N-1} = 1, \bar x_N = 0$. In this case, $a^\transp x \leq b$ reduces to $\sum_{i=1}^{2r+n+2} a_i x_i \leq \sum_{i=1}^{2r} t f_i$, and $\alpha^\transp  x \leq \beta$ is the same as $\sum_{i=1}^{2r+n+2} \alpha_i x_i \leq  \sum_{i=1}^{2r} f_i$. From Claim~\ref{claim: 4}, we have that $\alpha^\transp \bar x \leq \beta$ is always satisfied in this case. 
			\item[(b)] $\bar x_{N-1} = 1, \bar x_N = 1$. From $a^\transp \bar x \leq b$, we have
			\begin{multline*}
				\sum_{i=1}^{2r+1} t f_i \bar x_i + \big( t (2L+1)+1\big) \bar x_{2r+2} + (t+1) \sum_{i=1}^n w_i \bar x_{i+2r+2} \leq \sum_{i=1}^{2r} t f_i - t - 1.
			\end{multline*}
			Since $\bar{x}\in\{0,1\}^N$, we have $\sum_{i=1}^{2r+1}  f_i \bar x_i  \in \Z$. It implies that
			\begin{multline*}
				\sum_{i=1}^{2r+1} f_i \bar x_i \leq \sum_{i=1}^{2r} f_i -1 - (2L+1) \bar x_{2r+2}
				- \sum_{i=1}^n w_i \bar x_{i+2r+2} - \left\lceil \frac{1+\bar x_{2r+2} + \sum_{i=1}^n w_i \bar x_{i+2r+2}}{t}\right\rceil.
			\end{multline*}
			Hence, in this case we always have
			\begin{multline*}
				\alpha^\transp \bar x = \sum_{i=1}^{2r+1} f_i \bar x_i + (2L+2) \bar x_{i+2r+2}
				+ \sum_{i=1}^n w_i \bar x_{i+2r+2} + f_{2r+1} + t + 2L+1 \\
				\leq  \sum_{i=1}^{2r+1} f_i + \bar x_{2r+2} + t + 2L - \left\lceil \frac{1+\bar x_{2r+2} + \sum_{i=1}^n w_i \bar x_{i+2r+2}}{t}\right\rceil 
				\leq \sum_{i=1}^{2r+1} f_i  + t + 2L = \beta-1. 
			\end{multline*}
			\item[(c)] $\bar x_{N-1} = 0, \bar x_N = 0$. In this case, if $\bar x_{2r+2} = 0$, then $\alpha^\transp \bar x \leq \sum_{i=1}^{2r+1} f_i + \sum_{i=1}^n w_i < \beta$. So we assume $\bar x_{2r+2} = 1$. Then from $a^\transp \bar x \leq b$, we have\begin{displaymath}
				\sum_{i=1}^{2r+1} t f_i \bar x_i + (t+1) \sum_{i=1}^n w_i \bar x_{i+2r+2} \leq \sum_{i=1}^{2r+1} t f_i + t^2 + t.
			\end{displaymath}
			This implies
			\begin{multline}
				\sum_{i=1}^{2r+1} f_i \bar x_i \leq \sum_{i=1}^{2r+1} f_i + t+1
				-\sum_{i=1}^n w_i \bar x_{i+2r+2}-\left\lceil \frac{\sum_{i=1}^n w_i \bar x_{i+2r+2}}{t} \right\rceil.\label{eq:last_00}
			\end{multline}
			Depending on the value of $\sum_{i=1}^n w_i \bar x_{i+2r+2} $, we further consider the following three cases.
			\begin{itemize}
				\item[(c1)] If $\sum_{i=1}^n w_i \bar x_{i+2r+2} \leq t-1$, then $\alpha^\transp \bar x\leq\sum_{i=1}^{2r+2}\alpha_i+\sum_{i=1}^n w_i \bar x_{i+2r+2} \leq \sum_{i=1}^{2r+1} f_i + 2L+2 + t-1 = \beta$.
				\item[(c2)] If $\sum_{i=1}^n w_i \bar x_{i+2r+2} = t$, then consider the new point $\hat{x} \in \{0,1\}^N$ with $\hat{x}_i = 1$ for $i = 1, \ldots, 2r+2, \hat{x}_i = \bar{x}_i$ for $i = 2r+3, \ldots, 2r+n+2, \hat{x}_{N-1} = \hat{x}_N = 0$. We have $a^\transp \hat x = \sum_{i=1}^{2r+1} t f_i + t (2L+1) + 1 + t(t+1) = b$ while $\alpha^\transp \hat x = \sum_{i=1}^{2r+1} f_i + 2L+2 + t = \beta + 1$. So here $\hat{x}$ is in the KP but it does not satisfy the inequality $\alpha^\transp x \leq \beta$. 
				\item[(c3)] If $\sum_{i=1}^n w_i \bar x_{i+2r+2} \geq t+1$, then \eqref{eq:last_00} implies that $\sum_{i=1}^{2r+1} f_i \bar x_i \leq \sum_{i=1}^{2r+1} f_i + t-\sum_{i=1}^n w_i \bar x_{i+2r+2} - 1$. It implies that $\alpha^\transp \bar x = \sum_{i=1}^{2r+1} f_i \bar x_i + 2L+2 +  \sum_{i=1}^n w_i \bar x_{i+2r+2} \leq \beta$. 
			\end{itemize}
			Cases (c1)-(c3) imply that $\alpha^\transp x \leq \beta$ is valid for any binary point $\bar x$ satisfying $a^\transp \bar x \leq b$ and $\bar{x}_{N-1} = \bar{x}_N = 0$ if there does not exist $S \subseteq [n]$ such that $w(S) = t$, in which case (c2) does not happen. On the other hand, if there exists $S \subseteq [n]$ such that $w(S) = t$, then one can construct $\hat{x}$ like in case (c2) to show that $\alpha^\transp x\leq\beta$ is not valid. Therefore, $\alpha^\transp x \leq \beta$ is valid for any binary point $\bar x$ satisfying $a^\transp \bar x \leq b$ and $\bar{x}_{N-1} = \bar{x}_N = 0$ if and only if there does not exist $S \subseteq [n]$ such that $w(S) = t$.
			\item[(d)] $\bar x_{N-1} = 0, \bar x_N = 1$. In this case, if $\bar x_{2r+2} = 0$, then $\alpha^\transp \bar x \leq \sum_{i=1}^{2r+1} f_i + \sum_{i=1}^n w_i < \beta$. So we assume $\bar x_{2r+2} = 1$. Then $a^\transp \bar x \leq b$ reduces to
			\begin{displaymath}
				\sum_{i=1}^{2r+1} t f_i \bar x_i + (t+1) \sum_{i=1}^n w_i \bar x_{i+2r+2} \leq \sum_{i=1}^{2r+1} t f_i + t^2 - 1.
			\end{displaymath}
			This implies 
			\begin{multline}
				\sum_{i=1}^{2r+1} f_i \bar x_i \leq \sum_{i=1}^{2r+1} f_i + t
				-\sum_{i=1}^n w_i \bar x_{i+2r+2}-\left\lceil \frac{1+\sum_{i=1}^n w_i \bar x_{i+2r+2}}{t} \right\rceil.
				\label{eq:last_01}
			\end{multline}
			If $\sum_{i=1}^n w_i \bar x_{i+2r+2} \leq t-1$, then $\alpha^\transp \bar x \leq \sum_{i=1}^{2r+2}\alpha_i+\sum_{i=1}^n w_i \bar x_{i+2r+2}\leq \sum_{i=1}^{2r+1} f_i + 2L+2 + t-1 = \beta$. If $\sum_{i=1}^n w_i \bar x_{i+2r+2} \geq t$, then \eqref{eq:last_01} yields that $\sum_{i=1}^{2r+1} f_i \bar x_i  \leq \sum_{i=1}^{2r+1} f_i + t-\sum_{i=1}^n w_i \bar x_{i+2r+2}-2$. Therefore, $\alpha^\transp \bar x = \sum_{i=1}^{2r+1} f_i \bar x_i + 2L+2 + \sum_{i=1}^n w_i \bar x_{i+2r+2} \leq \beta -1$. 
		\end{enumerate}
		
		From the discussion of the above four cases, we know that for any binary point $\bar x \in \{x \in \{0,1\}^N: a^\transp x \leq b\}$ with $\bar{x}_{N-1} + \bar x_N \geq 1$, inequality $\alpha^\transp \bar x \leq \beta$ always holds. At the same time, inequality $\alpha^\transp x \leq \beta$ is valid for any binary point $\bar x$ satisfying $a^\transp \bar x \leq b$ and $\bar{x}_{N-1} = \bar{x}_N = 0$ if and only if there does not exist $S \subseteq [n]$ such that $w(S) = t$. We have thus concluded that $w(S) \neq t$ for any subset $S \subseteq [n]$ if and only if inequality $\alpha^\transp x \leq \beta$ is valid for $\{x \in \{0,1\}^N: a^\transp x \leq b\}$. 
		
		Lastly, we want to show that there exist $N$ affinely independent points in $\{x \in \{0,1\}^N: a^\transp x \leq b, \alpha^\transp x = \beta\}$ if and only if there exists subset $S \subseteq [n]$ such that $w(S) = t-1$. 
		
		From Claim~\ref{claim: 4}, there exist $v_1, \ldots, v_{N-2} \in \{0,1\}^{N-2}$ that are affinely independent, and they satisfy 
		$\sum_{i=1}^{N-2} a_i x_i \leq \sum_{i=1}^{2r} t f_i$ and $\sum_{i=1}^{N-2} \alpha_i x_i = \sum_{i=1}^{2r}  f_i$. Therefore, points $(v_1, 1, 0), \ldots, (v_{N-2}, 1, 0) \in \{0,1\}^N$ are affinely independent, and they satisfy 
		\begin{multline*}
			\sum_{i=1}^{N} a_i x_i \leq \sum_{i=1}^{2r} t f_i + a_{N-1}
			= \sum_{i=1}^{2r+1} t f_i + t^2 + t(2L+2) + 1 = b
		\end{multline*}
		and 
		\begin{displaymath}
			\sum_{i=1}^{N} \alpha_i x_i = \sum_{i=1}^{2r}  f_i + \alpha_{N-1} = \sum_{i=1}^{2r+1}  f_i + t + 2L + 1 = \beta.
		\end{displaymath}
		
		If there exists $x^* \in \{0,1\}^n$ such that $\sum_{i=1}^n w_i x^*_i = t-1$, then we can construct another two points $p,q$ as follows:
		\begin{align*}
			& p_i = 
			\begin{cases}
				1,        &i =1, \ldots, 2r+2,\\
				x^*_{i-2r-2},        &i = 2r+3, \ldots, 2r+n+2, \\
				0,  &i =2r+n+3,2r+n+4,
			\end{cases}\\
			&q = p + \bfe_N.
		\end{align*}
		Notice that
		\begin{align*}
			a^\transp p \leq a^\transp q = \sum_{i=1}^{2r+1} t f_i + t(2L+1) + 1
			+ \sum_{i=1}^n (t+1) w_i x^*_i + t+1 = b,\\
			\alpha^\transp p = \alpha^\transp q = \sum_{i=1}^{2r+1} f_i + 2L+2 + \sum_{i=1}^n w_i x^*_i = \sum_{i=1}^{2r+1} f_i + t + 2L +1 = \beta.
		\end{align*}
		So both points $p$ and $q$ satisfy $a^\transp x \leq b$ and $\alpha^\transp x = \beta$. It is easy to check that these $N$ points $(v_1, 1, 0), \ldots, (v_{N-2}, 1, 0), p, q$ are affinely independent points in $\{x \in \{0,1\}^N: a^\transp x \leq b, \alpha^\transp x = \beta\}$. 
		
		On the other hand, assume that there exist $N$ affinely independent points in $\{x \in \{0,1\}^N: a^\transp x \leq b, \alpha^\transp x = \beta\}$. Then there must exist a point $p^* \in \{x \in \{0,1\}^N: a^\transp x \leq b, \alpha^\transp x = \beta\}$ with $p^*_N = 1$, since otherwise $\{x \in \{0,1\}^N: a^\transp x \leq b, \alpha^\transp x = \beta\}$ will be contained in the hyperplane given by $x_N = 0$, which violates the assumption. Furthermore, here $p^*_{N-1} = 0$, because if $p^*_{N-1} = 1$, then point $p^*$ falls into the case (b) above, in which case we have $\alpha^\transp p^* \leq \beta - 1$, contradicting the assumption that $\alpha^\transp p^* = \beta$. Hence, we have $p^*_{N-1} = 0, p^*_N = 1$, and $p^*$ falls into the case (d) above. Following the same argument there, in order to have $\alpha^\transp p^* = \beta$, we must have $p^*_{2r+2} = 1$ and $\sum_{i=1}^n w_i p^*_{i+2r+2} \leq t-1$. If $\sum_{i=1}^n w_i p^*_{i+2r+2} \leq t-2$, then $\alpha^\transp p^* \leq \sum_{i=1}^{2r+1} f_i + 2L+2 + \sum_{i=1}^n w_i p^*_{i+2r+2} \leq \sum_{i=1}^{2r+1} f_i + t + 2L = \beta - 1$. Therefore, we must have $\sum_{i=1}^n w_i p^*_{i+2r+2} = t-1$, which means that there exists subset $S \subseteq [n]$ such that $w(S) = t-1$. 
		\smallskip
		
		All in all, we have established that: 
		\begin{itemize}
			\item[(i)] There does not exist $S \subseteq [n]$ such that $w(S) = t$ if and only if inequality $\alpha^\transp x \leq \beta$ is valid for the KP defined by $a^\transp x \leq b$;
			\item[(ii)] There exist $N$ affinely independent points in $\{x \in \{0,1\}^N: a^\transp x \leq b, \alpha^\transp x = \beta\}$ if and only if there exists subset $S \subseteq [n]$ such that $w(S) = t-1$. 
		\end{itemize}
		Combining (i) and (ii), we have shown that there is a ``yes" answer to the \texttt{CSS} problem $(w, t)$ if and only if $\alpha^\transp x \leq \beta$ is a facet-defining inequality to the KP $\conv(\{x \in \{0,1\}^N: a^\transp x \leq b\})$. 
		\hfill \Halmos
		
		\section{\texttt{KNAPSACK FACETS} on Inequalities with a Fixed Number of Distinct Coefficients}\label{sec: xp}
		In the previous section, we have shown that for a general inequality and a knapsack polytope, it is \DP-complete to recognize whether this inequality is facet-defining for the knapsack polytope. However, for an inequality with binary left-hand side coefficients, it is easy to determine whether such an inequality is facet-defining. The following theorem is rephrased from Theorem~1 by Balas \cite{balas1975facets}, and similar results were also shown independently by Hammer et al. \cite{hammer1975facet}, and Wolsey \cite{wolsey1975faces}. 
		\begin{theorem}[\cite{balas1975facets,hammer1975facet,wolsey1975faces}]
			\label{theo: balas_canonical}
			Whether an inequality of the form $\sum_{j \in M} x_j \leq k$ defines a facet of the KP $\conv(\{x\in\{0,1\}^n:a^\transp x\leq b\})$ can be determined in time $\texttt{poly}(n)$.
		\end{theorem}
		In this section, we show a more general result, that is, assuming that each arithmetic operation takes $O(1)$ time, determining whether an inequality $\alpha^\transp x\leq \beta$ with $(\alpha,\beta)\in\bZ^{n+1}_+$ is facet-defining for the knapsack polytope is \emph{slicewise polynomial} ({\XP}) with respect to $|\alpha|_+$, where $|\alpha|_+$ denotes the cardinality of the set $\{\alpha_i:\alpha_i>0,i\in[n]\}$, i.e., the number of distinct positive values that coefficients $\alpha_i$ are taking. Specifically, the problem of determining if $\alpha^\transp x \leq \beta$ is facet-defining for a knapsack polytope can be solved in time $n^{|\alpha|_+}\cdot\texttt{poly}(n)$.
		
		Throughout this section, let $Q$ denote the knapsack set $\{x\in\{0,1\}^n:a^\transp x\leq b\}$ and $P$ denote the associated KP $\conv(Q)$. Recall that $(a,b)\in\N^{n+1}$. Without loss of generality, we assume $\max_i a_i\leq b$, in which case both $Q$ and $P$ are full-dimensional.
		Then up to a positive scaling, any facet-defining inequality $\alpha^\transp x \leq \beta$ of $P$ that is different from the non-negativity constraints $x\geq 0$ must have $(\alpha,\beta)\in\bZ_+^{n+1}$ \cite{hammer1975facet}.
		We also assume that $\alpha^\transp x=\sum_{k=1}^K\gamma_k\sum_{i\in I_k}x_i$, where $K:=|\alpha|_{+}$, $\gamma\in\N^K$, $I_k=\{i_{k-1}+1,i_{k-1}+2,\ldots,i_k\}$ with $0=i_0< i_1<\ldots< i_{K}\leq n$, and vector $a$ is sorted such that $a_{i_{k-1}+1}\leq a_{i_{k-1}+2}\leq\ldots\leq a_{i_k}$ for $k=1,\ldots,K$, and $a_{i_K+1}\leq a_{i_K+2}\leq\ldots\leq a_n$.
		
		A useful tool we use for proving our {\XP} result is the concept of minimal basic knapsack solutions.
		\begin{definition}
			Given vector $z\in\bZ_+^K$ with $0\leq z_k\leq |I_k|$ for $k\in[K]$, we call the vector $x[z]\in\{0,1\}^n$, satisfying the following, the minimal $z$-basic knapsack solution (with respect to $\alpha$):
			\begin{itemize}
				\item $x_{i_{k-1}+1}[z]=x_{i_{k-1}+2}[z]=\ldots=x_{i_{k-1}+z_k}[z]=1$ for $k\in[K]$;
				\item $x_{i_{k-1}+z_k+1}[z]=x_{i_{k-1}+z_k+2}[z]=\ldots=x_{i_{k}}[z]=0$ for $k\in[K]$;
				\item $x_{i_{K}+1}[z]=x_{i_{K}+2}[z]=\ldots=x_{n}[z]=0$,
			\end{itemize}
			i.e., the first $z_k$ components of $x_{I_k}$ being $1$ for $k\in[K]$, and the rest being $0$.
		\end{definition}

		\begin{proposition}\label{prop:validity}
			Validity of the inequality $\alpha^\transp x\leq \beta$ with $(\alpha,\beta)\in\bZ^{n+1}_+$ for the knapsack polytope $P$ can be verified in time $n^{K+O(1)}$.
		\end{proposition}
		\textbf{Proof.}
		Assume that $\bar{x}\in\{0,1\}^n$ satisfies $a^\transp \bar{x}\leq b$ while $\alpha^\transp \bar{x}> \beta$. Define $\bar{z}\in\bZ_+^K$ such that $\bar{z}_k=\sum_{i\in I_k}\bar{x}_i$ for $k\in[K]$. Consider the $\bar{z}$-basic knapsack solution $x[\bar{z}]$. Since $a_{i_{k-1}+1}\leq a_{i_{k-1}+2}\leq\ldots\leq a_{i_k}$ for $k=1,\ldots,K$, we have $a^\transp x[\bar{z}]\leq a^\transp \bar{x}\leq b$ and $\alpha^\transp x[\bar{z}]=\alpha^\transp \bar{x}> \beta$. Therefore, if $\alpha^\transp x\leq \beta$ is not valid for $P$, then there exists a minimal $\bar{z}$-basic knapsack solution satisfying $a^\transp x[\bar{z}]\leq b$ and $\alpha^\transp x[\bar{z}]> \beta$. However, if $\alpha^\transp x\leq \beta$ is valid for $P$, then such solution does not exist. Therefore, verifying validity of the inequality $\alpha^\transp x\leq \beta$ amounts to checking through all minimal basic knapsack solutions, which takes time $O(n\prod_{k=1}^K(|I_k|+1))\leq n^{K+O(1)}$.
		\hfill \Halmos
		
		\begin{proposition}\label{prop:support}
			Assume inequality $\alpha^\transp x\leq \beta$ with $(\alpha,\beta)\in\bZ^{n+1}_+$ is valid for the knapsack polytope $P$. Then inequality $\alpha^\transp x\leq \beta$ defines a facet of $P$ if and only if the following two conditions hold:\begin{enumerate}
				\item There exists $\bar{x}\in Q$ such that $\alpha^\transp \bar{x}=\beta$ and $\bar{x}_n=1$.
				\item Inequality $\sum_{k=1}^K\gamma_k\sum_{i\in I_k}x_i\leq \beta$ is facet-defining for the knapsack polytope $P'=\conv(\{x\in\{0,1\}^{i_K}:\sum_{i=1}^{i_K}a_ix_i\leq b\})$.
			\end{enumerate}
			Moreover, the first condition can be checked in time $n^{K+O(1)}$.
		\end{proposition}
		\textbf{Proof.}
		Recall that $(a,b)\in\N^{n+1}$, $b\geq\max_i a_i$, and $P$ is full-dimensional.
		We first show that both conditions are necessary for inequality $\alpha^\transp x\leq \beta$ to be facet-defining for $P$. Assume inequality $\alpha^\transp x\leq \beta$ is facet-defining for $P$. Then the first condition must hold. Otherwise, by full dimensionality of $P$, $\alpha^\transp x\leq \beta$ can only be a multiple of $x_n\geq 0$, which contradicts the fact that $(\alpha,\beta)\in\bZ^{n+1}_+$. Inequality $\sum_{k=1}^K\gamma_k\sum_{i\in I_k}x_i\leq \beta$ is valid for $P'$ as $\alpha^\transp x\leq\beta$ is valid for $P$. Note that the face of $P'$ defined by $\sum_{k=1}^K\gamma_k\sum_{i\in I_k}x_i\leq \beta$ is the orthogonal projection of the face of $P$ defined by $a^\transp x\leq b$ to its first $n_K$ coordinates. It then follows that the second condition must hold as inequality $\alpha^\transp x\leq \beta$ is facet-defining for $P$.
		
		We next show that conditions 1 and 2 are sufficient for inequality $\alpha^\transp x\leq \beta$ to be facet-defining for $P$. Assume conditions 1 and 2 hold. Without loss of generality, we can assume $\bar{x}_{i_{K}+1}=\bar{x}_{i_{K}+2}=\ldots=\bar{x}_{n-1}=0$. (Otherwise, replacing those coordinates by $0$ yields such $\bar{x}$.) Since condition 2 holds, there exist affinely independent vectors $y^1,\ldots,y^{i_K}\in\{0,1\}^{i_K}$ on the facet of $P'$ defined by $\sum_{k=1}^K\gamma_k\sum_{i\in I_k}x_i\leq \beta$. Observe that the following $n$ vectors are affinely independent and lie on the face of $P$ defined by $\alpha^\transp x\leq \beta$:\begin{gather*}
			\bar{x}+\bfe_{i_{K}+1}-\bfe_n,\bar{x}+\bfe_{i_{K}+2}-\bfe_n,\ldots,\bar{x}+\bfe_{n-1}-\bfe_n,
			\bar{x},(y^1,\bzero_{n-i_K}),(y^2,\bzero_{n-i_K}),\ldots,(y^{i_K},\bzero_{n-i_K}).
		\end{gather*}
		Therefore, $\alpha^\transp x\leq \beta$ is facet-defining for $P$.
		
		Similar to the proof of Proposition \ref{prop:validity}, condition 1 holds if and only if there exists there exists vector $z$ and a minimal $z$-basic knapsack solution $x[z]$ such that $\bar{x}=x[z]+\bfe_n$ satisfies condition~1. Therefore, checking condition 1 amounts to checking through all minimal basic knapsack solutions, which takes time $n^{K+O(1)}$. \hfill \Halmos 
		
		We next extend the definition of minimal basic knapsack solutions to basic knapsack solutions.
		
		\begin{definition}
			For each $z\in\bZ_+^K$ satisfying $0\leq z_k\leq |I_k|$ for all $k\in [K]$ and $\sum_{k=1}^K\gamma_kz_k=\beta$, for all $k\in[K]$ with $0<z_k< |I_k|$, we call the following $|I_k|-1$ vectors the block-$k$ $z$-basic knapsack solutions:
			\begin{align*}
				&x^i[z,k]:=x[z]-\bfe_{i_{k-1}+i}+\bfe_{i_{k-1}+z_k+1},
				&&i=1,\ldots,z_k-1,\\
				&x^{i}[z,k]:=x[z]-\bfe_{i_{k-1}+z_k}+\bfe_{i_{k-1}+i},
				&&i=z_k+1,\ldots,|I_k|.
			\end{align*}
			We call the minimal $z$-basic knapsack solution $x[z]$ for all $k\in [K]$ and all block-$k$ $z$-basic knapsack solutions for all $k\in [K]$ with $0<z_k<|I_k|$, the $z$-basic knapsack solutions. We call a $z$-basic knapsack solution $x$ feasible if $a^\transp x\leq b$, and infeasible otherwise.
		\end{definition}
		
		We next show some basic properties of feasible $z$-basic knapsack solutions.
		\begin{lemma}\label{lem:basic_soln}
			Assume $z\in\bZ_+^K$ satisfies $0\leq z_k\leq |I_k|$ for all $k\in [K]$ and $\sum_{k=1}^K\gamma_kz_k=\beta$.
			Let $X[z]$ denote the set of all feasible $z$-basic knapsack solutions. Then the following hold:\begin{enumerate}
				\item All $z$-basic knapsack solutions are affinely independent of each other;
				\item If $z_k=0$ for some $k\in[K]$, then $X[z]$ is contained in the affine subspace defined by $x_{i_{k-1}+1}=\ldots=x_{i_{k}}=0$;
				\item If $z_k=|I_k|$ for some $k\in[K]$, then $X[z]$ is contained in the affine subspace defined by $x_{i_{k-1}+1}=\ldots=x_{i_{k}}=1$;
				\item If $0<z_k<|I_k|$ and $x^j[z,k]\notin X[z]$ (i.e., $x^j[z,k]$ is infeasible) for some $k\in[K]$ and $j\leq z_k-1$, then there does not exist $x\in Q$ such that $\sum_{i\in I_k}x_i=z_k$ for $k\in[K]$ and $x_{i_{k-1}+j}=0$, in which case set $X[z]$ is contained in the affine subspace defined by $x_{i_{k-1}+j}=1$;
				\item If $0<z_k<|I_k|$ and $x^j[z,k]\notin X[z]$ (i.e., $x^j[z,k]$ is infeasible) for some $k\in[K]$ and $j\geq z_k+1$, then there does not exist $x\in Q$ such that $\sum_{i\in I_k}x_i=z_k$ for $k\in[K]$ and $x_{i_{k-1}+j}=1$, in which case set $X[z]$ is contained in the affine subspace defined by $x_{i_{k-1}+j}=0$.
			\end{enumerate}
		\end{lemma}
		\textbf{Proof.}
		The first conclusion follows from the definition of $z$-basic knapsack solutions. The second and third conclusions are trivial. Fix $k\in [K]$ and $j\leq z_k-1$. Note that $(a_i)_{i\in I_{k'}}$ is assumed to be sorted such that $a_{i_{k'-1}+1}\leq a_{i_{k'-1}+2}\leq\ldots\leq a_{i_{k'}}$ for $k'\in[K]$. Therefore, for $x\in\{0,1\}^n$ such that $\sum_{i\in I_{k'}}x_i=z_{k'}$ for all $k'\in[K]$, and $x_{i_{k-1}+j}=0$, we have\begin{multline*}
			a^\transp x=\sum_{k'\in[K]}\sum_{i\in I_{k'}}a_ix_i
			\geq \sum_{k'\in[K]:k'\neq k}\sum_{j'=1}^{z_{k'}}a_{i_{k'-1}+j'}+
			\sum_{j'\in [z_k]:j'\neq j}a_{i_{k-1}+j'}+a_{i_{k-1}+z_k+1}=a^\transp x^j[z,k],
		\end{multline*}
		i.e., \begin{multline*}
			x^j[z,k]\in\arg\min\{a^\transp x:x\in\{0,1\}^n;
			\sum_{i\in I_k}x_i=z_k,~k\in[K];~x_{i_{k-1}+j}=0\}.
		\end{multline*}
		In this case, $x^j[z,k]$ being infeasible (i.e., $a^\transp x^j[z,k]>b$) implies that there is no feasible $x\in\{0,1\}^n$ such that $\sum_{i\in I_{k'}}x_i=z_{k'}$ for all $k'\in[K]$, and $x_{i_{k-1}+j}=0$. The fourth conclusion then follows. The fifth conclusion follows similarly. \hfill \Halmos 
		
		We are now prepared to show the main result of this section.
		\begin{theorem}\label{thm: xp}
			Determining whether inequality $\alpha^\transp x\leq \beta$ with $(\alpha,\beta)\in\bZ_+^{n+1}$ is facet-defining for the knapsack polytope $P=\conv(\{x\in\{0,1\}^n:a^\transp x\leq b\})$ can be done in time $n^{K+O(1)}$.
		\end{theorem}
		\textbf{Proof.}
		By Propositions \ref{prop:validity} and \ref{prop:support}, we can assume without loss of generality that $i_K=n$ and $\alpha^\transp x\leq \beta$ is valid for $P$. Also by Theorem \ref{theo: balas_canonical}, we only need to consider the case when $K\geq 2$. We will show that determining whether inequality $\alpha^\transp x\leq \beta$ is facet-defining for $P$ amounts to checking whether there exist $n$ affinely independent vectors among all feasible basic knapsack solutions. Note that there are no more than $O((n-K+1)\prod_{k=1}^K(|I_k|+1))\leq n^{K+O(1)}$ feasible basic knapsack solutions. Checking whether there exist $n$ affinely independent vectors among all feasible basic knapsack solutions can be done in time $n^{K+O(1)}$ by computing the rank of a matrix with $(n+1)$ rows and up to $n^{K+O(1)}$ columns.
		
		First, if there exist $n$ affinely independent vectors among all feasible basic knapsack solutions, then $\alpha^\transp x\leq \beta$ is facet-defining for $P$ by the definition of feasible basic knapsack solutions. Now, assume for contradiction that there does not exist $n$ affinely independent vectors among all feasible basic knapsack solutions, and $\alpha^\transp x\leq \beta$ is facet-defining for $P$. Then there exists $\bar{x}\in Q$ lying on the facet of $P$ defined by $\alpha^\transp x\leq \beta$ such that $\bar{x}$ is affinely independent of all feasible basic knapsack solutions. Define $\bar{z}\in\bZ_{+}^K$ such that $\bar{z}_k=\sum_{i\in I_k}\bar{x}_i$ for $k=1,\ldots,K$. Consider the set $X[\bar{z}]$ of all feasible $\bar{z}$-basic knapsack solutions. Note that $X[\bar{z}]=\emptyset$ if and only if $a^\transp x[\bar{z}] > b$, which implies $a^\transp\bar{x}\geq a^\transp x[\bar{z}] > b$, i.e., $\bar{x}\notin Q$. Therefore, $X[\bar{z}]\neq \emptyset$.  Since $\bar{x}$ is affinely independent of all feasible basic knapsack solutions, $\bar{x}$ is affinely independent of all points in $X[\bar{z}]$. Note that, by Lemma \ref{lem:basic_soln}, the affine hull of $X[\bar{z}]$ is defined exactly by the following $n+1-|X[\bar{z}]|$ equations:\begin{enumerate}
			\item[(i)] $\sum_{i\in I_k}x_i=\bar{z}_k$, for $k\in[K]$ with $0<\bar{z}_k<|I_k|$;
			\item[(ii)] $x_{i_{k-1}+1}=\ldots=x_{i_{k}}=0$, for $k\in[K]$ with $\bar{z}_k=0$;
			\item[(iii)] $x_{i_{k-1}+1}=\ldots=x_{i_{k}}=1$, for $k\in[K]$ with $\bar{z}_k=|I_k|$;
			\item[(iv)] $x_{i_{k-1}+j}=1$, for $k\in[K]$ with $0<\bar{z}_k<|I_k|$ and $j\leq \bar{z}_k-1$ with $x^j(\bar{z},k)\notin X[\bar{z}]$;
			\item[(v)] $x_{i_{k-1}+j}=0$, for $k\in[K]$ with $0<\bar{z}_k<|I_k|$ and $j\geq \bar{z}_k+1$ with $x^j(\bar{z},k)\notin X[\bar{z}]$.        
		\end{enumerate}
		Note that $\bar{x}$ already satisfies the equations defined by (i)-(iii). Since $\bar{x}$ is affinely independent of all points in $X[\bar{z}]$, one the following two must be true:\begin{enumerate}
			\item There exist $k\in[K]$ with $0<\bar{z}_k<|I_k|$ and $j\leq \bar{z}_k-1$ such that $x^j(\bar{z},k)\notin X[\bar{z}]$ while $\bar{x}_{i_{k-1}+j}=0$;
			\item Or there exist $k\in[K]$ with $0<\bar{z}_k<|I_k|$ and $j\geq \bar{z}_k+1$ such that $x^j(\bar{z},k)\notin X[\bar{z}]$ while $\bar{x}_{i_{k-1}+j}=1$.
		\end{enumerate}
		It then contradicts with the last two conclusions of Lemma \ref{lem:basic_soln} as $\bar{x}\in Q$. \hfill \Halmos 
		
		Extending the proof, one can show that the ``strength" of a valid inequality $\alpha^\transp x\leq \beta$ for defining the KP can be computed  in time $n^{K+O(1)}$.
		\begin{corollary}
			Given a valid inequality $\alpha^\transp x\leq \beta$ with $(\alpha,\beta)\in\bZ_+^{n+1}$, the dimension of the face of $P$ defined by $\alpha^\transp x\leq \beta$ can be computed in time $n^{K+O(1)}$.
		\end{corollary}
		\textbf{Proof.}
		The proof of Theorem \ref{thm: xp} implies that every $\bar{x}\in Q$ can be written as an affine combination of several feasible basic knapsack solutions. Therefore, computing $\dim(P)$ amounts to computing the maximum number of affinely independent vectors among all feasible basic knapsack solutions, which can be done in time $n^{K+O(1)}$ by computing the rank of a matrix with $(n+1)$ rows and up to $n^{K+O(1)}$ columns. \hfill \Halmos 
		
		\full{
			The proofs in this section can also be easily extended to show that
			determining whether an inequality $\alpha^\transp x\leq \beta$ is facet-defining for totally-ordered multidimensional knapsack polytope \cite{del2023multi} is {\XP} with respect to $|\alpha|_+$.}
		
		\full{
			\section{KP Membership Problem}
			\label{sec: kp_membership}
			In this section, we study the \emph{membership problem} associated with KP.
			Generally speaking, the \emph{membership problem} asks the following question: Given an element $x$ and a set, is $x$ contained in this set?
			As trivial as this problem might seem to be, in some cases this decision problem can be rather hard to answer. As a selective list of examples in literature,
			\cite{MR916001} show that the membership problem for the copositive cone, that is deciding whether or not a given matrix is in the copositive cone, is a \coNP-complete problem. \cite{MR3165055} also show that the membership problem for the completely positive cone is \NP-hard. Moreover, it is worth mentioning that, for any separation problem: ``Given point $x^*$, find an inequality from a given cutting-plane family that is violated by $x^*$, or show none exists", the decision part of it can be essentially viewed as a membership problem: ``Given a point $x^*$ and the cutting-plane closure defined by intersecting all inequalities from the family, is $x^*$ contained in this closure?"
			In combinatorial optimization, the membership problem often takes the form of:
			``Given a point and a combinatorial polytope, is the point contained in this polytope?" Here the polytope is defined as the convex hull of integer (or binary) points satisfying certain properties, since a linear description of the polytope will trivialize the membership problem. 
			For any polytope arising from combinatorial optimization that is defined as the convex hull of certain set of binary points, e.g., the TSP polytope, the clique polytope, the matching polytope etc., the corresponding membership problem is obviously in \NP, since if point $p$ in this polytope, then by Carath\'edory's theorem there exist at most $d+1$ binary points such that $p$ can be written as the convex combination of there. Here $d$ is the dimension of the polytope. \cite{papadimitriou1982complexity} have shown that the membership problem of the TSP polytope is in fact \NP-complete. The next theorem gives an analogous result for KP. 
			Here recall the well-known \emph{partition problem}:  Given $(w_1, \ldots, w_n) \in \Z^n_+$, does there exist a subset $S \subseteq [n]$, such that $w(S) = w([n] \setminus S)$? 
			
			\begin{theorem}
				The membership problem of KP is \NP-complete.
			\end{theorem}
			\textbf{Proof.}
			Let $(a_1, \ldots, a_n)$ be an input to an instance of the \emph{partition problem}, let $x^*:=(1/2, \ldots, 1/2)$ and the knapsack constraint given by $a^\transp x \leq a([n])/2$. Since the \emph{partition problem} is \NP-complete, we only need to show that there exists $S \subseteq [n]$ with $a(S) = a([n] \setminus S)$ if and only if $x^* \in \conv\left( \left\{x \in \{0,1\}^n: a^\transp x \leq a([n])/2 \right\} \right)$. 
			
			If there exists $S \subseteq [n]$ with $a(S) = a([n] \setminus S)$, then $a(S) = a([n] \setminus S) = a([n])/2$, so $\chi^{S}, \chi^{[n] \setminus S} \in \left\{x \in \{0,1\}^n: a^\transp x \leq a([n])/2 \right\}$. Hence,
			$$x^* = 1/2\chi^{S} + 1/2 \chi^{[n] \setminus S} \in \conv\left( \left\{x \in \{0,1\}^n: a^\transp x \leq a([n])/2 \right\} \right).$$ 
			On the other hand, if $x^* \in \conv\left( \left\{x \in \{0,1\}^n: a^\transp x \leq a([n])/2 \right\} \right)$, then we know $x^*$ can be written as the convex combination of several points in $\{0,1\}^n$ which all satisfy $a^\transp x = a([n])/2$ since $a^\transp x^* = a([n])/2$. The support of any one of these binary points will serve as a yes-certificate to the \emph{partition problem}. \hfill \Halmos \\
		}
		
		\section{Concluding Remarks}
		We have shown in this paper that deciding whether a given inequality is facet-defining for a knapsack polytope is \DP-complete, but can be done in time $n^{K}\cdot\texttt{poly}(n)$ where $K$ denotes the number of distinct coefficients in the equality. It is still open whether it can be done in time $f(K)\cdot\texttt{poly}(n)$ where $f$ is some computable function independent of $n$, i.e., whether the knapsack facet recognition problem is fixed-parameter tractable.
		
		\bibliographystyle{plain}

		\bibliography{cite}

	\end{document}